\documentclass[a4paper,11pt]{article}

\usepackage[francais,english]{babel}
\usepackage[utf8]{inputenc}   
\usepackage{amsmath,amsthm}
\usepackage[dvips]{graphicx}
\usepackage{amsfonts,amssymb}
\usepackage{geometry}
\usepackage{hyperref}
\usepackage{stmaryrd}
\usepackage{fancyhdr}
\usepackage{url}
\usepackage{dsfont}

\newtheorem*{conj*}{Conjecture}

\newtheorem*{ack}{Acknowledgements}
\newtheorem*{thm*}{Theorem}

\newtheorem{prop}{Proposition}[section]

\newtheorem{LM}{Lemma}[section]

\newtheorem{thm}{Theorem}[section]
\newtheorem{df}{Definition}[section]
\newtheorem{cor}{Corollary}[section]

\newtheoremstyle{pourlesremarques}{\topsep}{\topsep}{\normalfont}{}{\bfseries}{.}{ }{}
\theoremstyle{pourlesremarques}
\newtheorem{rem}{Remark}[section]
\newtheorem*{rem*}{Remark}
\newtheoremstyle{pourlesexemples}{\topsep}{\topsep}{\normalfont}{}{\bfseries}{.}{ }{}
\theoremstyle{pourlesexemples}

\renewcommand{\o}{\mathfrak{O}}
\newcommand{\e}{\epsilon}

\newcommand{\w}{\varpi}
\renewcommand{\d}{\delta}
\newcommand{\R}{\mathbb{R}}
\renewcommand{\l}{\lambda}
\newcommand{\C}{\mathbb{C}}

\newcommand{\N}{\mathbb{N}}
\newcommand{\M}{\mathcal{M}}
\newcommand{\Z}{\mathbb{Z}}
\newcommand{\1}{\mathbf{1}}
\newcommand{\sm}{\mathcal{C}^\infty}

\title {\textbf{On the local Bump-Friedberg $L$-function II}}
\author{Nadir MATRINGE\footnote{Nadir Matringe, Universit\'e de Poitiers, Laboratoire de Math\'ematiques et Applications,
T\'el\'eport 2 - BP 30179, Boulevard Marie et Pierre Curie, 86962, Futuroscope Chasseneuil Cedex. Email: nadirmatringe@outlook.fr}}

\begin{document}
\maketitle

\begin{abstract}
Let $F$ be a $p$-adic field with residue field of cardinality $q$. To each irreducible representation of $GL(n,F)$, 
we attach a local Euler factor $L^{BF}(q^{-s},q^{-t},\pi)$ via the Rankin-Selberg method, and show that it is equal 
to the expected factor $L(s+t+1/2,\phi_\pi)L(2s,\Lambda^2\circ \phi_\pi)$ of the Langlands' parameter $\phi_\pi$ 
of $\pi$. The corresponding local integrals were introduced in \cite{BF}, and studied in \cite{Mcrelle}. This work 
is in fact the continuation of \cite{Mcrelle}. The result is a consequence of the fact that if 
$\d$ is a discrete series representation of $GL(2m,F)$, and $\chi$ is a character of Levi subgroup $L=GL(m,F)\times GL(m,F)$ 
which is trivial on $GL(m,F)$ embedded diagonally, then $\d$ is $(L,\chi)$-distinguished if an only if it admits a Shalika model.
This result was only established for $\chi=\1$ before.   
\end{abstract}

\section{Introduction} 
\label{intro}

Let $F$ be a $p$-adic field with residue field of cardinality $q$. 
In \cite{Mcrelle}, we attached to any irreducible representation $\pi$ of $GL(n,F)$ and any character $\alpha$ of $F^*$,
 an Euler factor $L^{lin}(s,\alpha,\pi)$ (denoted $L^{lin}(\pi,\chi_\alpha,s)$ in [ibid.]). It is defined as the gcd of 
a family of local integrals $\Psi(s,\alpha,W,\Phi)$ for $W$ in the Whittaker model of (the induced representation of 
Langlands' type above) $\pi$ and $\Phi$ a Schwartz map on $F^{\lfloor (n+1)/2 \rfloor}$. These integrals were introduced in \cite{BF} 
where the corresponding global integrals were studied. More precisely, in \cite{BF}, the corresponding global 
integrals were considered as maps of two complex variables, whereas in \cite{Mcrelle}, the character 
$\alpha$ is fixed, and the integrals are viewed as maps of the complex variable $s$. In particular, 
if $\alpha$ is an unramified character $|.|^t$, writing $L^{lin}(s,t,\pi)$ for $L^{lin}(s,\alpha,\pi)$ and
 $\Psi(s,t,W,\Phi)$ for $\Psi(s,\alpha,W,\Phi)$, it is not obvious that the map $L^{lin}(s,t,\pi)$ is rational in $q^{-s}$ and $q^{-t}$. Here, 
we consider the integrals $\Psi(s,t,W,\Phi)$ as maps of the variables $s$ and $t$. 
We show that they belong to $\C(q^{-s},q^{-t})$, and admit a gcd in a certain sense, that we denote $L^{BF}(s,t,\pi)$, which is 
the inverse of an element of $\C[q^{-s},q^{-t}]$ with constant term equal to $1$. We show that $L^{BF}(s,t,\pi)$ admits a functional equation, and that it is equal to the Galois factor $L(s+t+1/2,\phi_\pi)L(2s,\Lambda^2\circ \phi_\pi)$ (Theorem \ref{equal}), where 
$\phi_\pi$ is the Langlands' parameter of $\pi$. 
This result follows from the results of \cite{Mcrelle} about $L^{lin}$, and from the following new ingredient which we now explain. 
Let $\d$ be a discrete series representation of $GL(2m,F)$, and $\chi$ is any character of the bloc 
diagonal Levi $L=GL(m,F)\times GL(m,F)$, 
trivial on $GL(m,F)$ embedded diagonally, then $\d$ is $(L,\chi)$-distinguished if an only if it admits a Shalika model 
(Theorem \ref{main}). This result was only proved when $\chi$ is trivial before. 

{\ack{I thank the referees for their useful comments. This work was partially supported by the research project ANR-13-BS01
-0012 FERPLAY.}}

\subsection{Notations}

We denote by $F$ a $p$-adic field, by $\o$ its ring of integers, by $\w$ its uniformiser, and by $q$ the cardinality of its residue field. We denote by $|.|$ the absolute value of $F$ normalised by $|\w|=q^{-1}$. We denote by 
$G_n$ the group $GL(n,F)$ of invertible elements of the algebra $\mathcal{M}(n,F)$ (which we denote by $\mathcal{M}_n$). As usual, 
we will consider $G_{n-1}$ as a subgroup of $G_n$ via the embedding $g\mapsto diag(g,1)$. We denote by $\nu(g)$ or simply 
by $|g|$ the positive real number $|det(g)|$ when $g\in G_n$. The group $A_n$ will be the diagonal torus of $G_n$, contained in the Borel subgroup $B_n$ of upper triangular matrices of $G_n$.  We denote by $N_n$ the unipotent radical of $B_n$ (the matrices of $B_n$ with $1$ on the diagonal). 
We denote by $P_n$ the mirabolic subgroup of $G_n$, i.e. the group of matrices with last row equal to 
$(0,\dots,0,1)$. We set $K_n=GL(n,\o)$. We denote by $w_n$ the element of the symmetric group $\mathfrak{S}_n$ naturally embedded in $G_n$, defined by 
$$\begin{pmatrix} 1 & 2 & \dots & m-1 & m & m+1 & m+2 & \dots & 2m-1 & 2m \\
 1 & 3 & \dots & 2m-3 & 2m-1 & 2 & 4 & \dots & 2m-2 & 2m  
 \end{pmatrix}$$ when $n=2m$ is even, and by 
$$\begin{pmatrix} 1 & 2 & \dots & m-1 & m & m+1 & m+2 &\dots & 2m & 2m+1 \\
 1 & 3 & \dots & 2m-3 & 2m-1 & 2m+1 & 2 & \dots & 2m-2 & 2m  
 \end{pmatrix}$$ when $n=2m+1$ is odd. We denote by $L_n$ the standard Levi subgroup of $G_n$ 
which is $G_{\lfloor (n+1)/2 \rfloor}\times G_{\lfloor n/2 \rfloor}$ embedded by the map $(g_1,g_2)\mapsto diag(g_1,g_2)$. If $\alpha$ 
is a smooth character of $F^*$, we denote by $\psi_\alpha$ the character of $L_n$ defined as 
$$\psi_\alpha:diag(g_1,g_2)\mapsto \alpha(det(g_1))/\alpha(det(g_2)).$$
We denote by $H_n$ the group $w_n^{-1} L_n w_n$, by $h(g_1,g_2)$ the matrix 
$w_n^{-1} diag(g_1,g_2)w_n$ of $H_n$ (with $diag(g_1,g_2)\in L_n$), and by $\chi_\alpha$ the character of $H_n$ defined as 
$$\chi_\alpha:h(g_1,g_2)\mapsto \alpha(det(g_1))/\alpha(det(g_2)).$$ Note that the groups $H_n$ are compatible in the sense 
that $H_n\cap G_{n-1}$ is naturally isomorphic to $H_{n-1}$, which will allow us to consider $H_{n-1}$ as a subgroup of $H_n$. 
We denote by $d_n$ the matrix $diag(1,-1,1,-1,\dots)$ of $G_n$, the group $H_n$ is the subgroup of $G_n$ fixed by the involution 
$g\mapsto d_n g d_n$. We denote by $\d_n$ the character 
$$\d_n=\chi_{|.|}: h(g_1,g_2)\mapsto |g_1|/|g_2|$$ of $H_n$, 
and we denote by $\chi_n$ (resp. $\mu_n$) the character of $H_n$ equal to $\d_n$ when $n$ is odd (resp. even), 
and trivial when $n$ is even (resp. odd). Hence ${\chi_n}_{|H_{n-1}}=\mu_{n-1}$, and ${\mu_n}_{|H_{n-1}}=\chi_{n-1}$. 
 If $C$ is a subset of $G_n$, we sometimes denote by $C^\sigma$ the set $C\cap H_n$. 
When $n=2m$ is even, we denote by $U_{(m,m)}$ the subgroup of $G_n$ of matrices 
$$u(x)=\begin{pmatrix} I_m & x \\ 0 & I_m\end{pmatrix}$$ for $x\in \mathcal{M}_m$, and we denote by $S_n$ the Shalika 
subgroup of $G_n$ of matrices of the form $u(x)diag(g,g)$, for $g\in G_m$ and $x\in \mathcal{M}_m$.

 We fix until the end of this work, a nontrivial smooth character $\theta$ of $(F,+)$. By abuse of notation, we also 
denote by $\theta$ the character of $N_n$ defined as 
$n \mapsto \theta(\sum_{i=1}^{n-1} n_{i,i+1})$. For $n$ even, we will denote by $\Theta$ the character of 
$S_n$ given by the formula 
$$\Theta(u(x)diag(g,g))=\theta(Tr(x)).$$ All the representations 
of closed subgroups of $G_n$ that we will consider will be smooth and complex. We will use the product notation 
of \cite{BZ} for normalised parabolic induction.

\subsection{Representations of Whittaker type}

Here, for convenience of the reader, we summarise Section 2.2 of \cite{Mcrelle}, 
which is a compilation of well-known results about 
representations of Whittaker type. Those results are extracted from 
\cite{BZ}, \cite{Z}, \cite{R}, \cite{JS1}, we refer to \cite{Mcrelle} for the details. 

\begin{df}\label{dsicreteseries}
Let $a<b$ be integers, and $r$ a positive integer, and set $n=r(b-a+1)$. If $\rho$ be a cuspidal representation of $G_r$, then the representation 
$\nu^a\rho \times \dots \times \nu^{b-1}\rho$ has a unique irreducible quotient which we denote by $\d([a,b],\rho)$. 
We call such a representation of $G_n$ a discrete series representation. For $d\in \N^*$, we write $\d(d,\rho)$ for 
$\d([1-d,0],\nu^{(1-d)/2}\rho)$.

\end{df}

If $\pi$ is a representation of $G_n$ admitting a central character, we denote it by $\omega_\pi$. For $\chi$ a character of 
$F^*$, we denote by $Re(\chi)$ the real number $r$ such that $|\chi|_\R=|.|^{r}$.

\begin{df}\label{whittype}
Let $\pi$ be a representation of $G_n$, such that $\pi$ is a product of discrete series $\d_1\times\dots\times \d_t$ of smaller linear groups, 
we say that $\pi$ is of Whittaker type.\\
If the discrete series $\d_i$ are ordered such that $Re(\omega_{\d_i})\geq Re(\omega_{\d_{i+1}})$, we say that $\pi$ is (induced) of Langlands' type. It follows from \cite{Sil} that $\pi$ has a unique irreducible quotient $L(\pi)$ (its \textit{Langlands' quotient}) which determines $\pi$, and that any irreducible representation of $G_n$ is the Langlands' quotient of a representation of Langlands' type.
\end{df}

We now define the Whittaker model of a representation of Whittaker type. We denote by $Ind$ the smooth induction functor, and by 
$ind$ the compact smooth induction functor.

\begin{prop}\label{fourretout}
Let $\pi$ be a representation of Whittaker type, then $Hom_{N_n}(\pi,\theta)$ is of dimension $1$, hence 
the space of intertwining operators $Hom_{G_n}(\pi,Ind_{N_n}^{G_n}(\theta))$ is of dimension $1$. The image of the 
(unique up to scaling) intertwining operator from $\pi$ to $Ind_{N_n}^{G_n}(\theta))$ is called the Whittaker model of $\pi$, we
 denote it by $W(\pi,\theta)$. When $\pi$ is of Langlands' type, the $G_n$-module $W(\pi,\theta)$ is isomorphic to $\pi$, and 
we set $W(L(\pi),\theta)=W(\pi,\theta)$, so this defines the Whittaker model of any irreducible representation.  
\end{prop}

\begin{rem}
Notice that with this standard definition, even those irreducible representations of $G_n$ which do not admit a Whittaker functional on their space, still have a Whittaker model. For example the Whittaker model of the trivial representation $\1_{G_n}$ is in fact that of 
$\nu^{(n-1)/2}\times \dots \times \nu^{(1-n)/2}$, the latter indeed admitting a Whittaker functional on its space.
\end{rem}

If $\pi$ is an irreducible representation of $G_n$, we denote by $\pi^\vee$ its (smooth) contragredient. 
If $\pi=\d_1\times \dots \times\d_t$ is a representation of Whittaker type of $G_n$, and $w$ is the anti-diagonal matrix of $G_n$ with only ones on the second diagonal, then $\widetilde{\pi}:g\mapsto \pi(^t\!g^{-1})$ is of Whittaker type, isomorphic to 
$\d_t^\vee\times \dots \times \d_1^\vee$. In particular, if $\pi$ is of Langlands' type, then $\widetilde{\pi}$ as well, and 
$L(\widetilde{\pi})=L(\pi)^\vee$. Moreover for $W$ in $W(\pi,\theta)$, then $\widetilde{W}:g\mapsto W(w^t\!g^{-1})$ belongs to 
$W(\widetilde{\pi},\theta^{-1})$.\\ 

The asymptotics of Whittaker functions in a representation of Whittaker type are controlled by the exponents of the derivatives of this representation. We refer to 3.5 of \cite{BZ} for the definition of the derivatives $\pi^{(k)}$ of a representation $\pi$ of $G_n$. 
If $\pi$ is of finite length, then these derivatives have finite length (see for example Proposition 2.5 of \cite{M-asympt}).

\begin{df}
Let $\pi$ be a representation of $G_n$ of finite length. We call the $(n-k)$-exponents of $\pi$ the central characters of 
the irreducible subquotients of a Jordan-Holder series of $\pi^{(n-k)}$.
\end{df}
 
The following result is extracted from the proof of Theorem 
2.1 of \cite{M-asympt} (see the ``stronger statement'' in [loc. cit.]).

\begin{prop}\label{DL}
 Let $\pi$ be a representation of $G_n$ of Whittaker type. For $k \in \{1,\dots,n\}$, let $(c_{k,i_k})_{i_k=1,\dots,r_k}$ be the family of 
$(n-k)$-exponents of $\pi$, then for every $W$ in $W(\pi,\theta)$, the map $W(z_1\dots z_n)$
 is a linear combination of functions of the form 
$$c_{\pi}(t(z_n))\prod_{k=1}^{n-1} c_{k,i_k}(t(z_k))|z_k|^{(n-k)/2}v(t(z_k))^{m_k}\phi_k(t(z_k)),$$ 
where $z_k=diag(t(z_k)I_{k},I_{n-k})$, for $i_k$ between $1$ and $r_k$, non negative integers $m_k$, and 
functions $\phi_k$ in $C_c^\infty(F)$. \end{prop}

\subsection{Local Langlands correspondence and local factors}\label{galoisfactors}

We refer to Section 7 of \cite{BH} for the vocabulary and assertions concerning the Weil-Deligne representations of 
the Weil group of $F$ and their local constants. 
Let $W_F$ be the Weil group of $F$. If $\phi$ is a semi-simple Weil-Deligne representation of $W_F$, we denote by 
$L(s,\phi)$ its Artin $L$-function, which satisfies $L(s,\phi_1\oplus\phi_2)=L(s,\phi_1)L(s,\phi_2)$ for any 
 semi-simple Weil-Deligne representations $\phi_1$ and $\phi_2$ of $W_F$. On the other hand, if $\pi$ and $\pi'$ are 
irreducible representations of $G_n$ and $G_{n'}$ respectively, we denote by $L(s,\pi,\pi')$ the local 
factor attached to the pair $(\pi,\pi')$ in \cite{JPS}. We will denote by $L(s,\pi)$ the factor $L(s,\pi,\1)$ 
where $\1$ is the trivial representation of the trivial group $G_0$. It is a theorem from 
\cite{HT} and \cite{H} that there is a bijection $\phi: \pi\mapsto \phi_\pi$ from the set 
of irreducible representations of $G_n$ (up to isomorphism) to the set of 
semi-simple Weil-Deligne representations of $W_F$ of dimension $n$ (up to isomorphism), which satisfies amongst other properties, that if 
 $\pi$ and $\pi'$ are 
irreducible representations of $G_n$ and $G_{n'}$ respectively, the one has $$L(s,\pi,\pi')=L(s,\phi_\pi\otimes \phi_{\pi'}).$$
 The map $\phi$ is called the Langlands correspondence, and if $\pi$ is an irreducible representation of $G_n$, we will say that 
$\phi_\pi$ is the Langlands' parameter of $\pi$. 
If $\phi$ is a semi-simple Weil-Deligne representation of $W_F$, we will denote by $\Lambda^2\circ \phi$ its exterior-square, 
which is again a semi-simple Weil-Deligne representation of $W_F$. If $\phi_1,\dots,\phi_t$ are semi-simple Weil-Deligne representations
 of $W_F$. Because 
$$\Lambda^2 \circ (\oplus_{i=1}^t \phi_i)= \oplus_{k=1}^t \Lambda^2 \circ \phi_k \oplus_{1\leq i<j \leq t} \phi_i\otimes \phi_j,$$  
we deduce the formula 
$$L(s,\Lambda^2 \circ (\oplus_{i=1}^t \phi_i))=\prod_{k=1}^t L(s,\Lambda^2 \circ \phi_k) \prod_{1\leq i<j \leq t}  L(s,\phi_i\otimes \phi_j).$$
Notice that it is well known that if $\pi=L(\d_1,\dots,\d_t)$ is an irreducible representation of $G_n$, then 
$$\phi_\pi=\oplus_{i=1}^t \phi_{\d_i}.$$

\section{Distinguished discrete series}

Let $H$ be a closed subgroup of $G_n$, and $\chi$ a character of $H$, we recall 
that a representation $\pi$ of $G_n$ is said to be $(H,\chi)$-distinguished if the space 
$Hom_H(\pi,\chi)$ is nonzero. In this section, which is the core of the paper, we show that if $\d$ is a discrete series representation 
of $G_n$ (with $n$ even), and $\psi$ is a character of $L_n$ trivial on the diagonal embedding of 
$G_{n/2}$ in $L_n$, then $\d$ is $(L_n,\psi)$-distinguished if and only if it admits a Shalika model (i.e. if 
$Hom_{S_n}(\d,\Theta)\neq 0$). One direction, namely $Hom_{S_n}(\d,\Theta)\neq 0 \Rightarrow Hom_{L_n}(\d,\psi_\alpha)\neq 0$, follows 
from the remark in Section 6.1 of \cite{JR}. The other direction was also known for $\psi$-trivial (the proof 
adapting easily to unitary $\psi$, see Remark \ref{dif}), the whole point is to extend the result to non unitary characters $\psi$.

\begin{prop}\label{cusp}
Let $\rho$ be a cuspidal representation of $G_r$, with $r$ a positive even integer, then $\rho$ is $(H_r,\chi_\alpha)$-distinguished if and only if 
it is $(S_r,\Theta)$-distinguished.
\end{prop}
\begin{proof}
We work with $L_r$ rather than $H_r$. Of course if $\rho$ is $(H_r,\chi_\alpha)$-distinguished, then it is $(L_r,\psi_\alpha)$-distinguished, so we take a nonzero element $L$ in 
$Hom_{L_r}(\rho,\psi_\alpha)$. From \cite{D}, Theorem 4.4, (ii) (see \cite{KT-cusp} for $\alpha=\1$), we know that for any $v$ 
in the space of $\rho$, the relative coefficient $$\psi_{L,v}:g\in G_r\mapsto L(\rho(g)v)$$ belongs to 
$\sm_c(L_r \backslash G_r, \psi_\alpha)$, and $v\mapsto \psi_{L,v}$ is a $G_r$-module injection of $\rho$ in $\sm_c(L_r \backslash G_r, \psi_\alpha)$. We set $m=r/2$. Using the Iwasawa decomposition 
$L_rU_{(m,m)}K_r$, we see that the map $$f_{L,v}:x\mapsto \psi_{L,v}(u(x))$$ belongs to $\sm_c(\M_{m})$. 
 We denote by $S_L$ the linear form on
$\delta$, defined by 
\begin{equation}\label{S} S_L(v)=\int_{x\in \M_{m}} f_{L,v}(x)\theta^{-1}(Tr(x))dx\end{equation} for $v\in \rho$. We claim that 
$S_L$ is a nonzero Shalika functional on 
$\rho$. Indeed, $S_L(v)=0$ for all $v$ in $\rho$ means that $\int_{x\in \M_{m}} f_{L,v}(x)\theta^{-1}(Tr(x))dx$ for all $v$ in $\rho$. Replacing $v$ by $diag(g,I_r)$ for 
$g\in G_r$, we deduce that 
$$\alpha(det(g))\int_{x\in \M_{m}} f_{L,v}(g^{-1}x)\theta^{-1}(Tr(x))dx$$
$$=|det(g)|^{m}\alpha(det(g))\int_{x\in \M_{m}}\!\!\! f_{L,v}(x)\theta^{-1}(Tr(gx))dx=0$$ for all $g\in G_r$. Hence the Fourier 
transform of $f_{L,v}$ is zero on $G_{m}$, hence on $\M_{m}$ by smoothness of $f_{L,v}$ and density of $G_{m}$ in  $\M_{m}$. 
So the map $f_{L,v}$ must be zero, 
and $f_{L,v}(I_r)=L(v)=0$, and this for any $v\in \rho$. It follows from simple change of variable in (\ref{S}) that $S_L$ is $\Theta$-invariant under $S_r$.

\end{proof}

\begin{cor}[of the proof]\label{cormult1}
Let $\rho$ be a cuspidal representation of $G_r$, then $Hom_{H_r}(\rho,\chi_\alpha)$ is of dimension at most $1$.
\end{cor}
\begin{proof}
We saw that the map $L\mapsto S_L$ in the proof above is injective. Our claim now follows from the uniqueness of Shalika functionals for irreducible representations of $G_r$ (\cite{JR}, Proposition 6.1).
\end{proof}

\begin{rem}\label{dif}
It is not clear that Theorem \ref{cusp} extends easily to discrete series as in \cite{M14}, for $\alpha \neq \mathbf{1}$. 
Indeed,the relative coefficients are not in $L^2(H_n\backslash G_n)$ as soon as $\alpha$ is not unitary
 (in fact if $C$ is an $(H_n,\chi_\alpha)$-relative coefficient, with $\alpha$ not unitary, then $|C|^2$ is not a map on $H_n\backslash G_n$).
\end{rem}

Before we study the case of discrete series, we also need to clear up a misunderstanding in Theorem 3.1 of \cite{M14}. In the statement 
of this theorem, the integer $n$ must be $\geq 2$ (it is also tacitly assumed in its proof). Indeed, for $n=1$, we have $H_1=G_1$, all characters of 
$G_1$ are cuspidal, so a character $\mu$ of $G_1$ can by $(H_1,\chi)$-distinguished, and this if and only if $\mu=\chi$.

\begin{prop}\label{intermediaire}
Suppose that $\rho$ is a cuspidal representation of $G_r$ ($r\geq 1$), $d$ is a positive integer, and set $n=dr$. Let 
$$\pi=\nu^{(1-d)/2}\rho\times \dots \times \nu^{(d-1)/2}\rho.$$ 
Then $dim_\C(Hom_{H_n}(\pi,\chi_\alpha))\leq 1$, and if $d$ is odd, $dim(Hom_{H_n}(\pi,\chi_\alpha))= 1$ implies 
that $\rho$ is $(H_r,\chi_\alpha)$-distinguished (in particular $r$ is even except if $r=1$, in which case $\rho=\alpha$). 
\end{prop}
\begin{proof}
The reader would benefit from reading Section 3 of \cite{Mcrelle} until the discussion before
 Theorem 3.14 of [ibid.] before reading this proof. We suppose that $\pi$ is distinguished. In particular, as $Z_n\subset H_n$, the central character $\omega_\rho$ is of order $d$, and 
$\rho$ is unitary. We set $M=M_{(r,\dots,r)}$ the standard Levi subgroup of $G=G_n$ such that  
$R=\nu^{(1-d)/2}\rho\otimes \dots \otimes \nu^{(d-1)/2}\rho$ is a representation of $M$, and $P=P_{(r,\dots,r)}$ the corresponding 
standard parabolic subgroup of $G$. We also set $H=H_n$, and $\chi=\chi_\alpha$. A system of representatives $R(P\backslash G/H)$ of 
the double quotient $P\backslash G/H$ 
is determined in Section 3.1 of \cite{Mcrelle}. To every $s$ in $R(P\backslash G/H)$, a standard parabolic subgroup 
$P_s\subset P$ of $G$, and its standard Levi subgroup $M_s$ is 
associated. We denote by $H_{M_s}$ the intersection $M\cap sHs^{-1}$, and by $\chi_s$ the character 
$\chi(s^{-1}.s)$ of $H_{M_s}$. Then, in the discussion before Theorem 3.2 of [loc. cit.], for 
any $(H,\chi)$-invariant linear $L$ form on $\pi$, and any $s$ in $R(P\backslash G/H)$, an $(H_{M_s},\chi_s)$-linear form $L_s$ is 
defined on 
the normalised Jacquet module $r_{M_s}^M R$, with the property that if $L_s$ is zero for every $s$ in $R(P\backslash G/H)$, then $L$ is zero. 
In particular, if $\pi$ is $(H,\chi)$-distinguished, and $L$ is a nonzero linear form in $Hom_H(\pi,\chi)$, then $L_s$ is nonzero for at least one $s$. As the 
representation $R$ of $M$ is cuspidal, this implies first that $M_s=M$. Thanks to Section 3.2 of \cite{Mcrelle}, it also implies 
that there are $l$ disjoint couples $i_k<j_k$ ($k=1,\dots,l$ with $2l\leq d$) in $\{1,\dots,d\}$, and natural 
integers $n_i^-$ such $n_i^+$ with $r=n_i^-+n_i^+$ for each $i\in I=\{1,\dots,d\}-\cup_{k=1}^l \{i_k,j_l\}$, such that 
$H_{M_s}$ is of the form $$\{diag(g_1,\dots,g_d)\in M, g_{i_k}=g_{j_k} \ for \ k=1,\dots,l, \ and \ g_i\in M_{(n_i^-,n_i^+)} \ for \ 
i \in I\}.$$ Moreover, thanks to Theorem 3.1 of \cite{M14}, when $r\geq 2$, for every $i\in I$, we must have $n_i^-=n_i^+$. As 
$R=\nu^{(1-d)/2}\rho\otimes \dots \otimes \nu^{(d-1)/2}\rho$ 
is $(H_{M_s},\chi_s)$-distinguished, and as $\rho$ is unitary, this implies that $l=d$ when $l$ is even, and $i_k=k$ and $j_k=d+1-k$ 
for each $k\in\{1,\dots,d/2\}$, and $l=d-1$ when $l$ is odd, $i_k=k$ and $j_k=d+1-k$ 
for each $k\in\{1,\dots,(d-1)/2\}$. Moreover if $l$ is odd and $r\geq 2$, then $r$ is even and 
$\rho$ is $(M_{(r/2,r/2)},\mu_\alpha)$-distinguished (i.e. $(H_r,\chi_\alpha)$-distinguished), and when $r=1$, then $\rho=\alpha$. 
In particular, in all cases, 
this implies that there is a unique $s\in R(P\backslash G/H)$ such that $L_s$ is nonzero, thus the map $L\mapsto L_s$ is injective. 
When $d$ is odd, then $L_s$ then lives in a $1$-dimensional space thanks to Corollary \ref{cormult1}, whereas when $d$ 
is even, the same assertion follows form Schur's Lemma. This proves the result.
\end{proof}

\begin{cor}\label{mult1}
Let $c$ be a positive integer, and $r$ be a positive integer. Let $\rho$ be a cuspidal representation of $G_r$, and 
$\d'=\delta(c,\rho)$. Then the representation $\tau=\nu^{-c/2}\d'\times\nu^{c/2}\d'$ satisfies 
$dim_\C(Hom_{H_n}(\tau,\chi_\alpha))\leq 1$.
\end{cor}
\begin{proof}
It is obvious, because setting $d=2c$, the representation 
$\tau$ is a quotient of $\pi=\nu^{(1-d)/2}\rho\times \dots \times \nu^{(d-1)/2}\rho$. The result now follows from Proposition 
\ref{intermediaire}.
\end{proof}

Let $n=n_1+\dots+n_t$, $\pi_i$ be an irreducible representation of $G_{n_i}$ for every $i$, and let 
$\pi=\pi_1\times \dots \times \pi_t$. According to 
Chapter 3 of \cite{GJ}, there is $r_\pi\in \R$ such that for any coefficient $f$ of $\pi$ (i.e. a map of the form $g\in G_n \mapsto <v^\vee,\pi(g)v>$ for $v\in V$ and 
$v^\vee \in V^\vee$), and any $\Phi\in \sm_c(\M_n)$, the integral $Z(s,\Phi,f)=\int_{G_n}f(g)\Phi(g)\nu(g)^s dg$ 
converges absolutely for $Re(s)>r_\pi$. These zeta integrals in fact belong to $\C(q^{-s})$, and
 span a fractional ideal of $\C[q^{\pm s}]$ containing $1$. One denotes by $L(s,\pi)$ the unique Euler factor which is a generator of this ideal, it satisfies the relation $L(s,\pi)=\prod_{i=1}^t L(s,\pi_i)$. Notice that 
this notation is coherent with that of Section \ref{galoisfactors}, as it is proved in Section 5 of \cite{JPS} that if $\pi$ 
is irreducible, the Godement-Jacquet factor $L(s,\pi)$ and the Rankin-Selberg factor $L(s,\pi,\1)$ are equal.   We now recall the following result from \cite{FJ}. 

\begin{prop}\label{implication1}
Let $2m=n=n_1+\dots+n_t$ be an even integer, and $\pi_i$ be an irreducible representation of $G_{n_i}$ for every $i$. If 
$\pi=\pi_1\times \dots \times \pi_t$ is such that $Hom_{S_n}(\pi,\theta)\neq 0$. Take a nonzero element $L$ 
of $Hom_{S_n}(\pi,\theta)$, and denote by $S_L(\pi,\Theta)$ the space of maps from $G_n$ to $\C$ of the form $S_{L,v}:g\mapsto L(\pi(g)v)$ for $v\in V$.  
Then there is $r\in \R$, such that for any $S$ in $S_L(\pi,\Theta)$ the integral $$I(s,\alpha,S)=\int_{G_m}S(diag(g,I_m))\alpha(det(g))\nu(g)^sdg$$ is absolutely convergent for $Re(s)>r$. Moreover, these integrals in fact belong to $\C(q^{-s})$, and
 span a fractional ideal of $\C[q^{\pm s}]$ equal to $L(s+1/2,\pi\otimes \alpha)\C[q^{\pm s}]$. In particular, 
the map $$\Lambda_{\pi,L}: v\mapsto I(0,\alpha,S_{L,v})/L(1/2,\pi\otimes\alpha)$$ is a nonzero element of $Hom_{L_n}(\pi,\psi_\alpha)$.
\end{prop}
\begin{proof}
The assumptions in \cite{FJ} are that $\pi$ is unitary and irreducible, and the proof is given in the archimedean case. 
However, we adapted their arguments to the $p$-adic case in \cite{M14}, but the absolute convergence of $I(S,s,\alpha)$ 
was not shown in this reference, because we were dealing with linear forms invariant under $S_n\cap P_n$ rather than $S_n$, hence we couldn't use the asymptotic expansion of Shalika functions from \cite{JR}. We thus give a 
proof for the case at hand, referring to \cite{M14} and \cite{JR}, where only minor modifications are needed. 

Let $S$ belong to $S_L(\pi,\Theta)$. First, an asymptotic expansion of the restriction of $S$ to the torus $A_n$ 
is given in \cite{JR}, Theorem 6.1. Their proof assumes $\pi$ irreducible, but all the arguments work for $\pi$ of finite length (as
 the Jacquet modules of $\pi$ are also of finite length). Hence, we get the absolute convergence of $I(s,\alpha,S)$, for 
$Re(s)$ larger than a certain real number $r$ independent of $S$, as in \cite{JR} after Lemma 6.1. We can now write for $Re(s)>r$, the equality $I(s,\alpha,S)=\sum_{k\in \Z} c_k(\alpha, S) q^{-ks}$, where $c_k(\alpha,S)=\int_{|g|=q^{-k}} S(diag(g,I_m))\alpha(det(g))dg$. The statement now follows from Proposition 4.2 of \cite{M14}, which is for $\alpha=\1$, but valid for any $\alpha$.
\end{proof}

We are now able to prove the main result of this section.

\begin{thm}\label{main}
Let $n=dr>0$ be an even integer, and $\rho$ be a cuspidal representation of $G_r$. The representation $\delta(d,\rho)$ is $(H_n,\chi_\alpha)$-distinguished if and only if it 
is $(S_n,\Theta)$-distinguished.
\end{thm}
\begin{proof}
First, we notice that if $\delta(d,\rho)$ is $(H_n,\chi_\alpha)$-distinguished or $(S_n,\Theta)$-distinguished, 
then its central character (which is a power of that of $\rho$) is trivial, so we assume that the $\omega_\rho$ is unitary so that 
$\d(k,\rho)$ is unitary for any $k$. By Proposition \ref{implication1}, if $\delta(d,\rho)$ is $(S_n,\Theta)$-distinguished, then it is 
$(H_n,\chi_\alpha)$-distinguished. For the converse, there are two cases to consider, which are proved differently. We thus suppose that 
$\delta(d,\rho)$ is $(H_n,\chi_\alpha)$-distinguished. 

If $d$ is odd (hence $r$ is even), as 
$\delta(d,\rho)$ is a quotient of $\nu^{(1-d)/2}\rho\times \dots \times \nu^{(d-1)/2}\rho$, then by Proposition \ref{intermediaire}, 
the representation $\rho$ is $(H_r,\chi_\alpha)$-distinguished, hence it is $(S_r,\Theta)$-distinguished according to Proposition 
\ref{cusp}. This in turn implies that $\delta(d,\rho)$ is $(S_n,\Theta)$-distinguished according to Theorem 6.1. of \cite{M14}. 
So this proves the theorem when $d$ is odd.

If $d=2c$ is even, we write $\d'=\delta(c,\rho),$ hence $\tau=\nu^{-c/2}\d'\times\nu^{c/2}\d'$ is $(S_n,\Theta)$-distinguished according to 
Proposition 3.8 of \cite{Mcrelle}, and we denote by $L$ a nonzero Shalika functional on $\tau$. It follows from Section 3 of \cite{Z}, that $\tau$ is of length $2$, with 
irreducible quotient $\delta=\delta(d,\rho)$, and a unique irreducible submodule that we denote by $\sigma$, which is the Langlands quotient of $\nu^{c/2}\d'\times\nu^{-c/2}\d'$. We suppose that the quotient $\delta$ is not $(S_n,\Theta)$-distinguished, and we will obtain a contradiction. In this situation $\sigma$ must be $(S_n,\Theta)$-distinguished, as $L_{|\sigma}$ is nonzero. Now, according to Proposition \ref{implication1}, $\Lambda_{\tau,L}$ is a nonzero linear form on $\tau$ which is $(H_n,\chi_\alpha)$-invariant, and 
$\Lambda_{\sigma,L_{|\sigma}}$ is a nonzero linear form on $\sigma$ which is $(H_n,\chi_\alpha)$-invariant. Finally, it is known 
that $L(s,\tau\otimes \alpha)=L(s,\sigma\otimes\alpha)$ (\cite{J-principal}, Theorem 3.4), hence the linear form $\Lambda_{\tau,L}$ extends $\Lambda_{\sigma,L_{|\sigma}}$, in 
particular $({\Lambda_{\tau,L}})_{|\sigma}= \Lambda_{\sigma,L_{|\sigma}}$ is nonzero. But $\d$ is $(H_n,\chi_\alpha)$-distinguished, in
particular there is a nonzero $(H_n,\chi_\alpha)$-linear form on $\d$, which can be seen as a nonzero $(H_n,\chi_\alpha)$-linear form 
on $\tau$ which vanishes on $\sigma$. This contradicts Proposition \ref{mult1}, and ends the proof. 
\end{proof}

\section{The local Bump-Friedberg $L$-function}

For simplicity, from now on, we will assume that $\alpha$ is an unramified character of the form $|.|^t$ for some $t\in \C$. This 
is the situation considered in \cite{BF}. We will denote by abuse of notation $\chi_t$ the character $\chi_{|.|^t}$ of $H_n$.

\subsection{Definition}

For $n\in \N-\{0\}$, we set $m=\lfloor (n+1)/2 \rfloor$. Let $\pi$ be a representation of $G_n$ of Whittaker type. 
Let $W$ belong to $W(\pi,\theta)$, and $\Phi$ belong to $\sm_c(F^m)$, 
$s$ and $t$ be complex numbers, and for $h(h_1,h_2)\in H_n$, we denote by $l_n(h(h_1,h_2))$ the bottom row of $h_2$ when $n$ is even, and the bottom row of $h_1$ when $n$ is odd. 
We consider the following integrals, the convergence of which will be addressed just after 
 $$\Psi(s,t,W,\Phi)=\int_{N_n^\sigma\backslash H_n}W(h)\Phi(l_n(h))\chi_{t}(h)\chi_{n}(h)^{1/2}|h|^sdh$$ and 
$$\Psi_{(0)}(s,t,W)=\int_{N_{n}^\sigma\backslash {P_n}^\sigma}W(h)\chi_{t}(h)\mu_{n}^{1/2}(h)|h|^{s-1/2}dh$$
$$=\int_{N_{n-1}^\sigma\backslash {G_{n-1}}^\sigma}W(h)\chi_{t}(h)\chi_{n-1}^{1/2}(h)|h|^{s-1/2}dh.$$

\begin{rem}
Notice the difference of notations with \cite{Mcrelle}. In [ibid.], we denote by 
$\Psi(W,\Phi,\chi_{|.|^t},s)$ the integral denoted by $\Psi(s,t,W,\Phi)$ here. 
\end{rem}

\begin{prop}\label{CV}
 Let $\pi$ be a representation of $G_n$ of Whittaker type. Let $W$ belong to 
$W(\pi,\theta)$, $\Phi$ belong to $\sm_c(F^m)$, and $\e_k=0$ when $k$ is even and $1$ when $k$ is odd.
 If for all $k$ in $\{1,\dots,n\}$ (resp. in $\{1,\dots,n-1\}$), we have $Re(s)>-[Re(c_{k,i_k})+\e_k Re(t+1/2)]/k$, the integral $\Psi(s,t,W,\phi)$ 
(resp. $\Psi_{(0)}(s,t,W)$) converges absolutely. It admits meromorphic extension to $\C\times \C$ as elements of $\C(q^{-s},q^{-t})$. 
\end{prop}
\begin{proof}
Let $B_n^\sigma$ be the standard Borel subgroup of $H_n$, and $\delta_{B_n}^\sigma$ be its modulus character. 
The integral $$\Psi(s,t,W,\Phi)$$ will converge absolutely 
as soon as the integrals $$\int_{A_n}W(a)\phi(l_n(a))\chi_n^{1/2}(a)\chi_{t}(a)|a|^s\d_{B_n^\sigma}^{-1}(a)d^*a$$ will 
do so for any $W\in W(\pi,\theta)$, and any $\Phi\in \sm_c(F^m)$. But, according to Proposition \ref{DL}, and writing $z$ as $z_1\dots z_n$, 
this will be the case if the integrals of the form 
 $$\int_{A_n} \!\!\!\!\!\! c_{\pi}(t(z_n))\Phi(l_n(z_n))\!\!\prod_{k=0}^{n-1} c_{k,i_k}(z_k)|z_k|^{(n-k)/2}v(t(z_k))^{m_k}\phi_k(t(z_k))
(\chi_n^{1/2}\chi_{t})(z)|z|^s\d_{B_n^\sigma}^{-1}(z)d^*z$$  converge absolutely. But 
$\d_{B_n^\sigma}^{-1}\chi_n^{1/2}(z_k)=|t(z_k)|^{-\frac{k(n-k)}{2}+\frac{\e_k}{2}}$, and $\chi_{t}(z_k)=|t(z_k)|^{\e_k t}$ where $\e_k=0$ if $k$ is even, and $1$ if $k$ is odd, 
hence this will be the case if 
$$\int_{F^*} c_{\pi}(x)\phi(0,\dots,0,x)|x|^{ns} |x|^{\e_n (t+1/2)} d^*x$$ and each integral 
$$\int_{F^*} c_{k,i_k}(x)|x|^{\e_k/2}v(x)^{m_k}\phi_k(x)|x|^{\e_k t}|x|^{ks} d^*x
=\int_{F^*} c_{k,i_k}(x)|x|^{ks}|x|^{\e_k (t+1/2)}v(x)^{m_k}\phi_k(x)d^*x$$ converges absolutely.
 This will be the case if for every $k\in\{1,\dots,n\}$, one has $$Re(s)+[Re(c_{k,i_k})+\e_k(Re(t+1/2)]/k>0.$$ 
 Moreover, this also shows that $\Psi(s,t,W,\phi)$ extends to $\C\times \C$ as an element of $\C(q^{-s},q^{-t})$, by 
Tate's theory of local factors for characters of $F^*$. The case of the integral $\Psi_{(0)}$ is similar.
\end{proof}

\begin{cor}[of the proof]\label{cor}
There is an element $Q(q^{-s},q^{-t})$ of $\C[q^{-s},q^{-t}]$ with $Q(0,0)=1$, such that $Q(q^{-s},q^{-t})\Psi(s,t,W,\phi)\in 
\C[q^{\pm s},q^{\pm t}]$ for all $W\in W(\pi,\theta)$ and $\phi\in \sm_c(F^m)$.
\end{cor}
\begin{proof}
It follows from the theory of Tate $L$-factors, that we can choose $Q$ to be a suitable power of the product of the 
Tate $L$-factors $L(c_{k,i_k},ks+\e_k (t+1/2))$.
\end{proof}

For convenience, we set $L_t=\C(q^{-t})$ and $O_t=\C[q^{-t}]$. 

\begin{prop}\label{df}
Let $\pi$ be a representation of $G_n$ of Whittaker type. The $L_t$-vector space spanned 
by the integrals $\Psi(s,t,W,\Phi)$ when $W$ and $\Phi$ vary in $W(\pi,\theta)$ and $\sm_c(F^m)$ is a fractional 
ideal of $L_t[q^{\pm s}]$ of the form $L^{BF}(s,t,\pi)L_t[q^{\pm s}]$, where 
$L^{BF}(s,t,\pi)$ is an Euler factor 
$1/P^{BF}(q^{-s},q^{-t},\pi)$ with $P^{BF}(q^{-s},q^{-t},\pi)\in L_t[q^{-s}]$ and $P^{BF}(0,q^{-t},\pi)=1$, which is uniquely determined. 
In fact, the $P^{BF}(q^{-s},q^{-t},\pi)$ belongs to 
$O_t[q^{- s}]=\C[q^{-s},q^{-t}]$, and for $W$ in $W(\pi,\theta)$ and $\Phi \in \sm_c(F^{m})$, the integral $P^{BF}(q^{-s},q^{-t},\pi)\Psi(s,t,W,\Phi)$ belongs to $O_t[q^{\pm s}]$.
\end{prop}
\begin{proof}
For $h_0\in H_n$, changing $\Psi(s,t,W,\Phi)$ by $\Psi(s,t,\rho(h_0)W,\rho(h_0)\Phi)$ (where $\rho$ denotes right translation) multiplies 
$\Psi(s,t,W,\Phi)$ by $\chi_n(h_0)^{-1/2}\chi_{t}(h_0)^{-1}|h_0|^{-s}$. This implies that the integrals $\Psi(s,t,W,\Phi)$ span a fractional 
ideal of $L_t[q^{-s}]$. Moreover, it is shown in the proof of Proposition 4.11. of \cite{Mcrelle}, that one can chose 
$W$ and $\Phi$ such that $\Psi(s,t,W,\Phi)=1$ for all $s$ and $t$. This implies the existence of $P^{BF}(q^{-s},q^{-t},\pi)$. 
Now, thanks to Corollary \ref{cor}, $P^{BF}(q^{-s},q^{-t},\pi)$ divides the polynomial $Q\in O_t[q^{-s}]$ in the ring 
$L_t[q^{\pm s}]$. As both have constant term equal to $1$, this first implies that $P^{BF}$ divides $Q$ in $L_t[q^{-s}]$, and as 
$O_t[q^{-s}]$ is a unique factorisation domain (because $O_t=\C[q^{-t}]$ is), this also implies that $P^{BF}$ belongs to $O_t[q^{-s}]$.
 Finally, for all $W$ and $\Phi$ we have $P^{BF}(q^{-s},q^{-t})\Psi(s,t,W,\Phi)\in L_t[q^{\pm s}]$, so we can write it $\sum_{k=-N}^N a_k(q^{-t})q^{-ks}$ for some $N\geq 0$, with $a_k\in \C(X)$. According to the proof of Corollary \ref{cor}, the integral $\Psi(s,t,\Phi,W)$ has no singularities 
of the form $\{s\in \C\}\times \{t=t_0\}$ (because the same holds for the Tate $L$-factors $L(c_{k,i_k},ks+\e_k (t+1/2))$), in particular the $a_k$'s belong to $O_t$. 
\end{proof}

\subsection{Results on a specialisation of the local Bump-Friedberg factor}

This specialisation in the variable $t$ (we fix $t$), which is an Euler factor with respect to $q^{-s}$, denoted $L^{lin}(s,t,\pi)$ is studied in \cite{Mcrelle}. As we will see, it is not really 
defined as a specialisation, but rather as the gcd of a the family of specialised integrals, and for this reason, it 
is not obvious at all that $L^{lin}(s,t,\pi)$ is rational 
as a map of $q^{-s}$ and $q^{-t}$. Hence, we will use results about this specialisation as an intermediate step, but they 
will not appear in our final statements. In fact, in general, we will see that except maybe for a finite number of
 values of $t$, the equality $L^{lin}(s,t,\pi)=L^{BF}(s,t,\pi)$ is true (the equality is probably true for all values of $t$, 
but we did not check it, however, it is true for discrete series representations).\\
 
We start by recalling a special case of Proposition 4.11 of \cite{Mcrelle}. 

\begin{prop}\label{defLlin}
Let $n$ be a positive integer, and $\pi$ be a representation of $G_n$ of Whittaker type. 
For fixed $t\in \C$, the integrals $\Psi(s,t,W,\phi)$ generate a (necessarily principal) 
fractional ideal $I(t,\pi)$ of $\C[q^{\pm s}]$ when $W$ and $\phi$ vary in their respective spaces, and $I(t,\pi)$ 
has a unique generator which is an Euler factor which we denote by $L^{lin}(s,t,\pi)$ (and denoted by 
$L^{lin}(\pi,\chi_{t},s)$ in \cite{Mcrelle}).\\
The integrals $\Psi_{(0)}(s,t,W)$ (which are equal to $1$ be convention if $n=1$) generate a (necessarily principal) 
fractional ideal $I_{(0)}(t,\pi)$ of $\C[q^s,q^{-s}]$ when $W$ and $\phi$ vary in their respective spaces, 
and $I_{(0)}(t,\pi)$ has a unique generator which is an Euler factor, which we denote by $L^{lin}(s,t,\pi)$ (and denoted by 
$L_{(0)}^{lin}(\pi,\chi_{t},s)$ in \cite{Mcrelle}). 
\end{prop} 

For $t\in \C$, we set $$L^{lin,(0)}(s,t,\pi)=L^{lin}(s,t,\pi)/L_{(0)}^{lin}(s,t,\pi)$$ (denoted 
$L_{rad(ex)}^{lin}(\pi,\chi_{t},s)$ in \cite{Mcrelle}). The following property is a consequence of Proposition 4.9 and 
Corollary 4.2 of \cite{Mcrelle}.

\begin{prop}
Let $n$ be a positive integer, and $\pi$ be a representation of $G_n$ of Whittaker type. Fix $t\in \C$, then $L^{lin,(0)}(s,t,\pi)$ is an Euler factor with simple poles. 
\end{prop}

We also recall Theorem 5.2 of \cite{Mcrelle}.

\begin{thm}\label{inductivity}
Let $n$ be a positive integer, and $\pi=L(\tau)$ be an irreducible representation of $G_n$, with $\tau=\d_1\times \dots \times \d_r$ induced of Langlands' type. 
Then, for $t\in [-1/2,0]$, we have $$L^{lin}(s,t,\pi)= \prod_{i=1}^r L^{lin}(s,t,\d_i)\prod_{1\leq i<j \leq r} L(2s, \d_i,\d_j) .$$
\end{thm}

\subsection{Equality with the Galois factor for discrete series}\label{discrete}

We recall a consequence of Proposition 4.14 of \cite{Mcrelle} in the special case of discrete series. As a convention, when 
$n=0$, and $\pi$ is the trivial representation of $G_0$, we set $$L^{lin}(s,t,\pi)=1.$$

\begin{prop}\label{rec}
Let $r$ and $d$ be positive integers (and set $n=dr$), and $\rho$ be a cuspidal representation of $G_r$. Then 
one has $$L^{lin}(s,t,\d(d,\rho))=L^{lin,(0)}(s,t,\d(d,\rho))L^{lin}(s,t,\nu^{1/2}\d(d-1,\rho)).$$
\end{prop}

In order to compute the factor $L^{lin,(0)}(s,t,\d(d,\rho))$, we recall Corollary 4.3 of \cite{Mcrelle}.

\begin{prop}\label{expolediscrete} Let $n$ be a positive integer and $t\in \C$. If $\d$ be a discrete series representation of $G_n$, then the factor 
$L^{lin,(0)}(s,t,\d)$ has a pole at $s=0$ if and only if $\d$ is $(H_n,\chi_{t}^{-1})$-distinguished.
\end{prop}

Now, we have the following consequence of Theorem \ref{main}.

\begin{prop}\label{eq1discrete}
Let $n\geq 2$ be a positive integer, and $\d$ be a discrete series representation of $G_n$. Then the factor 
$L^{lin,(0)}(s,t,\d(d,\rho))$ is equal to $L^{lin,(0)}(s,0,\d(d,\rho))$ for any $t\in \C$.
\end{prop}
\begin{proof}
Both are Euler factors with simple poles, hence it suffices to prove that they have the same poles. The result follows at once 
from Theorem \ref{main} and Proposition \ref{expolediscrete}.
\end{proof}

We can now prove the following.

\begin{thm}\label{eq2discrete}
Let $n$ be a positive integer and $t\in \C$, and $\d$ be a discrete series representation of $G_n$. Then the factor 
$L^{lin}(s,t,\d)$ is equal to $L(s+t+1/2,\phi_\d)L(2s,\Lambda^2\circ \phi_\d)$.
\end{thm}
\begin{proof}
Write $\d$ as $\d(d,\rho)$, for $d\in \N-\{0\}$ dividing $n$, and $\rho$ a cuspidal representations of $G_r=G_{n/d}$. If $r\geq 2$, 
by Propositions \ref{rec} and \ref{eq1discrete}, we obtain the equalities 
$$L^{lin}(s,t,\d(d,\rho))=\prod_{k=0}^{d-1}L^{lin,(0)}(s,t,\nu^{k/2}\d(d-k,\rho))$$ 
$$=\prod_{k=0}^{d-1}L^{lin,(0)}(s,0,\nu^{k/2}\d(d-k,\rho))=L^{lin}(s,0,\d(d,\rho)).$$
Now, by Theorem 5.4 of \cite{Mcrelle}, we have $L^{lin}(s,0,\d)=L(s+1/2,\phi_\d)L(2s,\Lambda^2\circ \phi_\d)$, but
 $L(s+1/2,\d)$ is equal to $1$ according to \cite{GJ}, so we obtain the equality 
$$L^{lin}(s,t,\d)=L(s+t+1/2,\d)L(2s,\Lambda^2\circ \phi_\d)$$ when $r\geq 2$.

When $r=1$, by Proposition \ref{rec}, we have $$L^{lin}(s,t,\d(d,\rho))=\prod_{k=0}^{d-1}L^{lin,(0)}(s,t,\nu^{k/2}\d(d-k,\rho)).$$ 
The factor $L^{lin,(0)}(s,t,\nu^{(d-1)/2}\d(1,\rho))$ is equal to 
$$L^{lin,(0)}(s,t,\nu^{(d-1)/2}\rho)=L^{lin}(s,t,\nu^{(d-1)/2}\rho)=L(s+t+d/2,\rho),$$ so 
 $$(\star)\ \ L^{lin}(s,t,\d(d,\rho))=L(s+t+d/2,\rho)\prod_{k=0}^{d-2}L^{lin,(0)}(s,t,\d(d-k,\rho)).$$ 
For $k\leq d-2$, the factor $L^{lin,(0)}(s,t,\d(d-k,\rho))$ do not depend on $t$ by Proposition \ref{eq1discrete}. As 
we know from Theorem 5.4 of \cite{Mcrelle} that $$L^{lin}(s,0,\d)=L(s+1/2,\phi_\d)L(2s,\Lambda^2\circ \phi_\d),$$ and 
as $L(s+1/2,\d(d,\rho))=L(s+d/2,\rho)$ we 
deduce, letting $t=0$ in Equality $(\star)$, that $L(2s,\Lambda^2\circ \phi_\d)$ must be equal to $$\prod_{k=0}^{d-2}L^{lin,(0)}(s,t,\nu^{k/2}\d(d-k,\rho)),$$ and 
this proves the result in this case.
\end{proof}

We will see in the next paragraph that this implies that $L^{BF}(s,t,\d)$ is equal to $L(s+t+1/2,\phi_\d)L(2s,\Lambda^2\circ \phi_\d)$ 
for any discrete series $\d$ of $G_n$ (and in fact for any irreducible representation).

\subsection{Equality with the Galois factor for irreducible representations}

We first notice a consequence of the results of the preceding section. 
If $\pi$ is an irreducible representation of $G_n$, we denote by $L^{Gal}(s,t,\pi)=L(s+t+1/2,\phi_\pi)L(2s,\Lambda^2\circ \phi_\pi)$.

\begin{prop}\label{eq1}
Let $\pi$ be an irreducible representation of $G_n$, then for $(s,t)\in \C \times [-1/2,0]$, we have 
$$L^{lin}(s,t,\pi)=L^{Gal}(s,t,\pi).$$
\end{prop}
\begin{proof}
 We saw in Section \ref{galoisfactors} 
that if $\pi=L(\d_1\times\dots\times\d_r)$, then 
$$L^{Gal}(s,t,\pi)=\prod_{k=1}^r L^{Gal}(s,t,\d_i) \prod_{1\leq i<j\leq r}L(2s,\phi_{\d_i}\otimes\phi_{\d_j}).$$
The result now follows from Theorems \ref{inductivity} and \ref{eq2discrete}.
\end{proof}
 
We will use the following elementary fact about polynomials in two variables.

\begin{LM}\label{elementarylemma}
Let $P$ be a polynomial in $\C[X,Y]$, and suppose that it does not vanish on a set of the form $\C\times A$, where 
$A$ is infinite, then $P$ is of the form $P(X,Y)=P(0,Y)$.
\end{LM}
\begin{proof}
Write $P$ as $P=\sum_{k=0}^d a_k(Y)X^k$, then for any $y\in A$, by the fundamental theorem of algebra, one has $P(X,y)=a_0(y)$. 
As $A$ is infinite, the result follows.
\end{proof}

\begin{prop}\label{div1}
Let $\pi$ be an irreducible representation of $G_n$, for $n\geq 1$. 
Then the polynomial 
$$P^{BF}(s,t,\pi)=\frac{1}{L^{BF}(s,t,\pi)}$$ divides the polynomial 
$$P^{Gal}(s,t,\pi)=\frac{1}{L(s+t+1/2,\phi_\pi)L(2s,\Lambda^2\circ \phi_\pi)}$$ in $\C[q^{-s},q^{-t}]$.
\end{prop}
\begin{proof}
For any $W$ in $W(\pi,\theta)$ and $\Phi\in \sm_c(F^m)$, the Laurent polynomial $R(s,t)=P^{Gal}(s,t,\pi)\Psi(s,t,W,\Phi)\in L_t[q^{\pm s}]$ 
has no singularities for $(s,t)$ in 
$\C \times [-1/2,0]$, according to Proposition \ref{eq1}. 
We can always write $R$ as a quotient $U/V$, with $U$ and $V\in \C[q^{-s},q^{-t}]$, and co-prime. In particular, 
$U$ and $V$ have a finite number of zeroes in common (it follows for example from B\'ezout identities in $L_s[q^{-t}]$ and 
$L_t[q^{-s}]$). In particular, $V$ does not vanish on a set of the form $\C\times A$, where 
$A$ is infinite (take $A$ the complementary set of the projection on the second coordinate of the set of common zeroes of $U$ and $V$), 
and is thus of the form $V=V(q^{-t})$ according to Lemma \ref{elementarylemma}. This implies that $R$ belongs to 
$L_t[q^{-s}]$, and by definition of $P^{BF}(s,t,\pi)$, it in turn implies that $P^{BF}(s,t,\pi)$ divides 
$P^{Gal}(s,t,\pi)$ in $L_t[q^{-s}]$, hence in $O_t[q^{-s}]$ as they both have constant term equal to $1$.
\end{proof}

Let $\pi$ be an irreducible representation of $G_n$, for fixed $t$, we set $$P^{lin}(s,t,\pi)=\frac{1}{L^{lin}(s,t,\pi)}.$$

\begin{prop}\label{div2}
For fixed $t$, $P^{lin}(s,t,\pi)$ divides $P^{BF}(s,t,\pi)$ in $\C[q^{-s}]$.
\end{prop}
\begin{proof}
Fix $t=t_0$. According to the last part of Proposition \ref{df}, for any $W$ in $W(\pi,\theta)$ and 
any $\Phi\in \sm_c(F^{\lfloor (n+1)/2 \rfloor})$ the element $$P^{BF}(q^{-s},q^{-t_0},\pi)\Psi(s,t_0,W,\Phi)$$ belongs to 
$\C[q^{\pm s}]$, and the proposition follows.
\end{proof}

We can now state the main result of this section.

\begin{thm}\label{equal}
Let $\pi$ be an irreducible representation of $G_n$, then one has the equality $$L^{Gal}(s,t,\pi)=L^{BF}(s,t,\pi).$$
\end{thm}
\begin{proof}
For $(s,t)\in \C\times [-1/2,0]$, we have $P^{Gal}(s,t,\pi)=P^{lin}(s,t,\pi)$ by Proposition \ref{eq1}. By Propositions \ref{div1}
 and \ref{div2}, we know that for fixed $t$, $P^{lin}(s,t,\pi)$ divides $P^{BF}(s,t,\pi)$, and $P^{BF}(s,t,\pi)$ 
divides $P^{Gal}(s,t,\pi)$. In particular, for $(s,t)\in \C\times [-1/2,0]$, $P^{BF}(s,t,\pi)$ must be equal to $P^{Gal}(s,t,\pi)$, 
so this equality is in fact true for all $s$ and $t$.
\end{proof}

We can also prove the following almost everywhere equality (which should be true everywhere).

\begin{prop}\label{almostequal}
Let $\pi$ be an irreducible representation of $G_n$. Then there is a finite (maybe empty) set $A_{\pi}$ of $\C^*$, such that 
for $q^{-t}\in \C^*-A_{\pi}$, the two factors 
$L^{lin}(s,t,\pi)$ and $L^{BF}(s,t,\pi)$ are equal in $\C(q^{-s})$.
\end{prop}
\begin{proof}
By definition of $L^{BF}(s,t,\pi)$, there is a finite set $I$, Whittaker maps $W_i\in W(\pi,\theta)$, Schwartz functions $\Phi_i\in 
\sm_c(F^{\lfloor (n+1)/2 \rfloor})$, and rational maps $\l_i\in L_t$, such that 
$$L^{BF}(s,t,\pi)=\sum_{i\in I} \l_i(t)\Psi(s,t,W_i,\Phi_i).$$ 
Let $A_{\pi}$ be the union of the possible poles of the $\l_i$'s, it is a finite set, and for $q^{-t_0}\in \C^*-A_{\pi}$, 
we can specialise the equality and obtain that $L^{BF}(s,t_0,\pi)$ belongs to 
the $\C[q^{\pm s}]$-submodule of $L_s$ spanned by the integrals $\Psi(s,t,W,\Phi)$ for $W\in W(\pi,\theta)$ and $\Phi\in 
\sm_c(F^{\lfloor (n+1)/2 \rfloor})$, hence $L^{BF}(s,t_0,\pi)$ divides $L^{lin}(s,t_0,\pi)$. This together with Proposition \ref{div2} implies 
that $L^{BF}(s,t_0,\pi)=L^{lin}(s,t_0,\pi)$, which is our claim.
\end{proof}

We can finally obtain the functional equation of the local Bump-Friedberg $L$-factor. For $\Phi\in 
\sm_c(F^{\lfloor (n+1)/2 \rfloor})$, we denote by $\widehat{\Phi}^\theta$ the Fourier transform of $\Phi$ with respect to a $\theta$-self-dual Haar 
measure of $F^{\lfloor (n+1)/2 \rfloor}$.

\begin{cor}
Let $\pi$ be an irreducible representation of $G_n$. Then there is a unit $\epsilon(s,t,\pi,\theta)$ 
of $\C[q^{\pm s}, q^{\pm t}]$, such that for all $W\in W(\pi,\theta)$ and $\Phi\in\sm_c(F^{\lfloor (n+1)/2 \rfloor})$, one has 
\begin{equation} \label{eqfct} \frac{\Psi(1-s,-1/2-t,\widetilde{W},\widehat{\Phi}^\theta)}{L^{BF}(1-s,-1/2-t,\pi^\vee)}=
\frac{\epsilon(s,t,\pi,\theta) \Psi(s,t,W,\Phi)}{L^{BF}(s,t,\pi)}.\end{equation}
\end{cor}
\begin{proof}
According to Proposition 4.11 of \cite{Mcrelle} and Proposition \ref{almostequal}, there is a finite (maybe empty)
set $A_\pi$ of $\C^*$, such that 
for $q^{-t}\in \C^*-A_{\pi}$, there is a unit $\epsilon_t(s,\pi,\theta)$ of $\C[q^{\pm s}]$ 
such that for all $W\in W(\pi,\theta)$ and $\Phi\in\sm_c(F^{\lfloor (n+1)/2 \rfloor})$, one has 
$$\frac{\Psi(1-s,-1/2-t,\tilde{W},\widehat{\Phi}^\theta)}{L^{BF}(1-s,-1/2-t,\pi^\vee)}=
\frac{\epsilon_t(s,\pi,\theta) \Psi(s,t,W,\Phi)}{L^{BF}(s,t,\pi)}.$$ 
As both quotients $$\frac{\Psi(1-s,-1/2-t,\tilde{W},\widehat{\Phi}^\theta)}{L^{BF}(1-s,-1/2-t,\pi^\vee)}$$ and 
$$\frac{\Psi(s,t,W,\Phi)}{L^{BF}(s,t,\pi)}$$ belong to $\C[q^{\pm s},q^{\pm t}]$, the map 
$\epsilon_t(s,\pi,\theta)$ extends to an element $\epsilon(s,t,\pi,\theta)$ in $\C(q^{-s},q^{-t})$ such that 
Equation (\ref{eqfct}) is satisfied. It remains to show that $\epsilon(s,t,\pi,\theta)$ is a unit 
of $\C[q^{\pm s}, q^{\pm t}]$. For $q^{-t}\notin A_\pi$, there is $\alpha_t \in \C^*$, and $m_t\in \Z$, such that 
$\epsilon(s,t,\pi,\theta)=\alpha_t q^{m_t s}$, and this implies that $\alpha_t$ extends to an element $\alpha(t)$ 
of $\C(q^{-t})$ and that $m_t$ does not depend on $t$. Suppose that $\alpha(t_0)=\infty$, this would imply that 
$\Psi(s,t_0,W,\Phi)/L^{BF}(s,t_0,\pi)$ is equal to $0$ for all $W\in W(\pi,\theta)$ and $\Phi\in\sm_c(F^{\lfloor (n+1)/2 \rfloor})$, but 
as we can choose them such that $\Psi(s,t_0,W,\Phi)=1$, this would imply that $P^{BF}(s,t_0,\pi)=0$. This would 
in turn imply, according to Proposition \ref{div1}, that $P^{Gal}(s,t_0,\pi)=0$, which is absurd. Hence $\alpha(t)$ has no pole, and 
we prove in the exact same manner that it has no zeroes. This implies that $\alpha(t)$ is of the form $\alpha q^{-lt}$ for some $l\in \Z$, 
and we are done. 
\end{proof}

\end{document}